\documentclass[a4paper,12pt,reqno]{amsart}
\usepackage{amssymb}
\usepackage{amsmath}
\usepackage{amscd}
\usepackage{stmaryrd}
\usepackage
{diagrams}

\def\i{\,\lrcorner\,}

\def\D{\Delta}

\def\ra{\rightarrow}

\def\.{\cdot}

\def\n{\nabla}
\def\nb{\nabla}
\def\l{\lambda}
\def\t{\tilde}
\def\beq{\begin{equation}}
\def\eeq{\end{equation}}
\def\bi{\begin{enumerate}}
\def\ei{\end{enumerate}}
\def\bea{\begin{eqnarray*}}
\def\eea{\end{eqnarray*}}
\def\ba{\begin{array}}
\def\ea{\end{array}}

\def\f{\varphi}

\def\o{\omega}

\def\e{\varepsilon}
\def\L{\Lambda}

\def\s{\sigma}

\def\r{\end{proof}}

\def \R{\mathbb{R}}
\def \RM{\mathbb{R}}

\def \N{\mathbb{N}}

\def \C{\mathbb{C}}

\def\End{{\rm End}}

\def\9{[\! (}
\def\0{)\! ]}
\def\[{\pmb{[}}
\def\]{\pmb{]}}

\def\Ric{\mathrm{Ric}}

\def\Id{\mathrm{Id}}
\def\be{\begin{equation}}

\def\ee{\end{equation}}

\def\GL{\mathrm{GL}}
\def\Gl{\mathrm{GL}}
\def\CO{\mathrm{CO}}
\def\SO{\mathrm{SO}}
\def\Sym{\mathrm{Sym}}
\def\Hol{\mathrm{Hol}}

\def\R{\mathbb{R}}

\newtheorem{ede}{Definition}[section]

\newtheorem{epr}[ede]{Proposition}
\newtheorem{prop}[ede]{Proposition}
\newtheorem{thrm}[ede]{Theorem}

\newtheorem{ath}[ede]{Theorem}

\newtheorem{elem}[ede]{Lemma}

\newtheorem{ecor}[ede]{Corollary}

\newtheorem{defi}[ede]{Definition}
\theoremstyle{definition}
\newtheorem{definition}[ede]{Remark}
\def\obs{\begin{definition}}
\def\eobs{\end{definition}}
 
\title{Weyl-parallel forms, conformal products and Einstein-Weyl manifolds}
\author{Florin Belgun, Andrei Moroianu}

\address{Florin Belgun\\ Institut f\"ur Mathematik\\ Humboldt-Universit{\"a}t
 zu Berlin\\ Unter den Linden 6\\ D-10099 Berlin, Germany}
 \email{belgun@math.hu-berlin.de}

\address{Andrei Moroianu \\ CMLS\\ {\'E}cole Polytechnique \\ UMR 7640
du CNRS \\ 91128 Palaiseau \\ France} \email{am@math.polytechnique.fr}

\thanks{This work was partially supported by the French-German cooperation
  project Procope no. 17825PG\\
The first named author was equally supported by the
Schwerpunktprogramm 1154 {\em Globale Differentialgeometrie} of the
DFG}

\begin{document}

\begin{abstract}

Motivated by the study of Weyl structures on conformal manifolds
admitting parallel weightless forms, we define the notion of conformal
product of conformal structures and study its basic properties. We
obtain a classification of Weyl manifolds carrying parallel forms, and
we use it to investigate the holonomy of the adapted
Weyl connection on conformal products. As an application we describe 
a new class of Einstein-Weyl manifolds of dimension 4.

 \bigskip

\noindent
2000 {\it Mathematics Subject Classification}: Primary 53A30, 53C05,
53C29.

\medskip
\noindent{\it Keywords:} Conformal products, conformally parallel
forms, Weyl structures, reducible holonomy.
\end{abstract}

\maketitle

\section{Introduction}

A {\em conformal structure} on a smooth manifold $M$ is an equivalence
class $c$ of Riemannian metrics modulo conformal rescalings, or,
equivalently, a positive definite symmetric bilinear tensor with
values in the square of the weight bundle $L$ of $M$. In contrast to
the Riemannian situation, there is no canonical connection on a
conformal manifold. Instead of the Levi-Civita connection, one can
nevertheless consider the affine space of torsion-free connections
preserving the conformal structure, called {\em Weyl structures}.

The fundamental theorem of conformal geometry states that this space
is in one-to-one correspondence with the space of connections on the
weight bundle $L$, and is thus modeled on the vector space of smooth
1-forms. It is worth noting that not every Weyl structure is (locally)
the Levi-Civita connection of a Riemannian metric in the conformal
class. This actually happens if and only if the corresponding
connection on $L$ has vanishing curvature, in which case the Weyl
structure is called {\em closed}. Every conformal problem involving
closed Weyl structures is locally of Riemannian nature, so we will be mainly
concerned with the case of non-closed Weyl structures.

Spin conformal manifolds with Weyl structures $D$ carrying parallel
spinors have been studied in \cite{bsmf}.  The basic idea, which
allows the reduction of the problem to the Riemannian case, is that
the curvature tensor of a non-closed Weyl structure is no longer
symmetric by pairs. This fact eventually shows that the spin holonomy
representation of a non-closed Weyl structure has no fixed points,
except in dimension 4, where genuine local examples do actually exist.

We consider here the analogous question for exterior forms:
Characterize (locally) those conformal manifolds $(M,c)$ which carry
an exterior form $\o$ parallel with respect to some Weyl structure
$D$.  If $D$ is closed, it is (locally) the Levi-Civita connection of
some metric $g\in c$ and $D$-parallel forms correspond to fixed points
of the Riemannian holonomy representation on the exterior bundle. By
the de Rham theorem, and the fact that the space of fixed points of a
tensor product representation is just the tensor product of the
corresponding spaces of each factor, one may assume that the holonomy
acts irreducibly on $TM$. In this case, the Berger-Simons theorem
provides the list of possible holonomy groups, so the problem reduces
to an algebraic (although far from being trivial) computation.

Back to the conformal setting, we remark that we can restrict
ourselves to the case of {\em weightless} forms since otherwise the Weyl
structure would be automatically closed. By choosing a Riemannian
metric $g\in c$, the equation $D\o=0$ becomes
\beq\label{cpf}\n^g_X\o=\theta \wedge X\i\o-X^\flat\wedge \theta
^\sharp\i\o\qquad\forall\ X\in TM,\eeq where $\n^g$ is the Levi-Civita
covariant derivative of $g$ and $\theta$ is the connection form of $D$
in the trivialization of $L$ determined by $g$. Exterior forms
satisfying \eqref{cpf} are called {\em locally conformal parallel
forms} in \cite{cabrera} and are shown to define, under some further
conditions, harmonic sections of the corresponding sphere bundles.

We start by remarking that a nowhere vanishing exterior $p$-form $\o$
($0<p<\dim(M)$) can not be parallel with respect to more than one Weyl
structure. In fact $\o$ defines a unique ``minimal" Weyl structure
$D^\o$ which is the only possible candidate for having $D\o=0$. We
next apply the Merkulov-Schwachh\"ofer classification of torsion-free
connections with irreducible holonomy \cite{SM} to the Weyl structure
$D^\o$.  A quick analysis of their tables shows that the possible
(non-generic) holonomy groups of irreducible Weyl structures are all
compact (except in dimension 4, where the solutions to our problem
turn out to correspond to Hermitian structures -- see Lemma
\ref{lck}). But, of course, a Weyl structure with compact (reduced)
holonomy is closed since its holonomy bundle defines (local)
Riemannian metrics which are tautologically $D$-parallel.

It remains to study the reducible case, which, unlike in the
Riemannian situation, is more involved. First of all, we extend the de
Rham theorem to the conformal setting.  To do this, we need to define
the notion of {\em conformal products}.  Indeed, in contrast to
Riemannian geometry, there is no canonical conformal structure on a
product $M_1\times M_2$ of two conformal manifolds $(M_1^{n_1},c_1)$
and $(M_2^{n_2},c_2)$ induced by the two conformal structures
alone. The algebraic reason is, of course, that the group
$\CO(n_1)\times \CO(n_2)\subset \GL(n_1+n_2,\R)$ is not included in
$\CO(n_1+n_2)$.

On the other hand, a property characterizing the Riemannian product
$(M,g)$ of two Riemannian manifolds $(M_1,g_1)$ and $(M_2,g_2)$ is the
existence of two {\em complementary orthogonal} Riemannian submersions
$p_i:(M,g)\to (M_i,g_i)$ (here, {\em complementary} means that $TM$ is
the direct sum of the kernels of $dp_i$ and {\em orthogonal} means
that these kernels are orthogonal at each point). Generalizing this to
conformal geometry, a conformal structure on the manifold
$M:=M_1\times M_2$ is said to be {\em a conformal product} of
$(M_1,c_1)$ and $(M_2,c_2)$ if the canonical submersions $p_1:M\ra
M_1$ and $p_2:M\ra M_2$ are {\em orthogonal} conformal submersions.

In Section 4 we show that every conformal product carries a unique
{\em adapted} reducible Weyl structure $D$ preserving the two factors,
and conversely, every reducible Weyl structure induces a local
conformal product structure.

The similarities with the Riemannian case stop here, however, since
the factors of a conformal product do not carry canonical Weyl
structures (in fact the restrictions of the adapted Weyl structure $D$
to each slice $\{x_1\}\times M_2$ or $M_1\times \{x_2\}$ of the
conformal product $M_1\times M_2$ depends on $x_1$ and $x_2$), so it
is not possible to interpret the space of $D$-parallel forms on
$M_1\times M_2$ in terms of the two factors.

On the other hand, the lack of symmetry of the curvature tensor of $D$
mentioned above, allows us to show (in Section 5) that every parallel
form on a conformal product with non-closed adapted Weyl structure is
of pure type and eventually has to be the weightless volume form of
one of the factors, exception made of the 2-dimensional conformal products 
and of some conformal products of dimension 4 (which are described 
in detail in Section 6).

As an application, we present an explicit construction of new families of
Einstein-Weyl structures in dimension 4, using conformal products of
surfaces by means of {\em bi-harmonic} functions. These reduce to the
well-known examples of hyper-Hermitian surfaces constructed by Joyce
\cite{j} in the particular case where the bi-harmonic function is the real
part of a holomorphic function on $\C^2$, examples that fully cover
the cases of conformal {\em multi-products} in dimension larger than
2, {\em i.e.}, of a Weyl structure leaving invariant more than one pair of
orthogonal proper subspaces. 

Note that the conformal product Ansatz already occurred (although
without being named) in the study of 3-dimensional Einstein-Weyl
structures \cite{toda}. The Einstein-Weyl structure on this
  conformal product is, however, different from the adapted one. 

\smallskip

A similar question about holonomy on $n$-dimensional conformal
manifolds $(M,c)$ was studied by S. Armstrong in \cite{arm}. He
considers the holonomy of the {\em 
  canonical Cartan connection}, which is not an affine connection on
$TM$ but a (uniquely defined) linear connection on an $n+2$-rank
vector bundle of Lorentzian signature, and classifies the occurring
holonomy groups. A decomposition theorem is also given in this
context, if the Cartan connection leaves a $k$-dimensional subspace
invariant, for $2\le k\le n$: Summarized in the terms of our present
paper, $(M,c)$ turns out to be, in this case, a closed conformal
product with Einstein factors (plus some relation between the scalar
curvatures).  

We see therefore that despite the obvious geometric particularity
of a general conformal product structure, no restriction on the
holonomy of the Cartan connection (also called {\em conformal
  holonomy}) is implied. This fact brings us to the first
of the following open questions about conformal products: 
\bi
\item Is there any invariant characterization of the underlying
  conformal structure of a conformal product? 
\item Can the multiple conformal products ({\em i.e.}, admitting more than
  one pair of orthogonal conformal submersions) be characterized
  geometrically? 
\item Which are the possible reduced holonomies of a Weyl structure?
\ei

On the other hand, as the results of \cite{toda} and Section 6
suggest, the conformal product Ansatz is expected to have further
applications.

\section{Preliminaries}

Let $M^n$ be a manifold and let $P$ denote the principal bundle of
frames. The {\em weight bundle} of $M$ is the real line bundle $L$
associated to $P$ via the representation $|\det| ^{1/n}$ of
$\Gl(n,\RM)$. More generally one can define the $k$-weight bundle
$L^k$ for every $k\in \RM$, associated to $P$ via the representation
$|\det| ^{k/n}$. Obviously $L^k\otimes L^p\cong L^{k+p}$ and
$L^{-n}\cong |\Lambda ^n(T^*M)|=\delta M$, which is the bundle of {\em
densities} on $M$, a trivial line bundle (associated to $P$ via the
representation $|\det|^{-1}$) even if $M$ is not orientable. Positive
densities are geometrically meaningful as ``absolute values'' of
volume forms and positive global densities induce Lebesgue-like
measures on $M$ (like the well-known {\em Riemannian volume element}
of a Riemannian manifold).

As the weight bundles are powers of $\delta M$, the notion of
positivity is still well-defined; more precisely, a section of $L^k$
is positive if it takes values in $P\times _{|\det|
^{k/n}}\RM^+\subset L^k$. A {\em weighted tensor} on $M$ is a section
of $TM^{\otimes a}\otimes T^*M^{\otimes b} \otimes L^k$ for some
$a,b\in\N$ and $k\in \RM$. Its weight is by definition the real number
$a-b+k$.

\begin{defi}\label{co}
A {\em conformal structure} on $M$ is a symmetric positive definite
bilinear form $c$ on $TM\otimes L^{-1}$, or, equivalently, a symmetric
positive definite bilinear form on $TM$ with values in $L^2$.
\end{defi}

A conformal structure on $M$ can also be seen as a reduction
$P(\CO_n)$ of $P$ to the conformal group $\CO_n\cong\RM^+\times{\rm
O}_n\subset \GL(n,\RM)$.

We denote by $\L^k_0M:=\L^kM\otimes L^k$ the bundle of weightless
exterior forms of degree $k$.  The conformal structure defines an
isomorphism between weightless vectors and $1$-forms: $TM\otimes
L^{-1}\cong \L^1_0M$. The scalar product $c$ on $\L^1M\otimes L$
induces a scalar product, also denoted by $c$, on the bundles
$\L^k_0M$. Moreover, the exterior product maps $\L^k_0M\otimes
\L^{n-k}_0M$ onto $\L^n_0M\cong\RM$, thus defining the Hodge operator
$*:\L^k_0M\to \L^{n-k}_0M$ by
\beq\label{ho}\o\wedge\*\sigma=c(\o,\sigma),\qquad\forall \o,\sigma\in
\L^k_0M.\eeq

There is a one-to-one correspondence between positive sections $l$ of
$L$ and Riemannian metrics on $M$, given by the formula
\beq\label{cg}c(X,Y)=g(X,Y)l^2,\qquad\forall\ X,Y\in TM.\eeq

\begin{defi}\label{ws}
A {\em Weyl structure} on a conformal manifold is a torsion-free
connection on $P(\CO_n)$.
\end{defi}

By (\ref{ho}), the Hodge operator is parallel with respect to every
Weyl structure.

\begin{ath}\label{thw} {\em (Fundamental theorem of Weyl geometry)} 
There is a one-to-one correspondence between Weyl structures and
covariant derivatives on $L$.
\end{ath}

\begin{proof} Every connection on $P$ induces a covariant derivative $D$ on
  $TM$ and on $L$. The connection is a Weyl structure if and only if
  $D$ satisfies $D_XY-D_YX=[X,Y]$ and $D_Xc=0$ for all vector fields
  $X,\ Y$ on $M$. Like in the Riemannian situation, these two
  relations are equivalent to the Koszul formula
  \be\begin{split}\label{Koszul}
  2c(D_XY,Z)=&D_X(c(Y,Z))+D_Y(c(X,Z))-D_Z(c(X,Y))\cr
  &+c([X,Y],Z)+c([Z,X],Y)+c([Z,Y],X),\end{split}\ee for all vector
  fields $X,\ Y,\ Z$ on $M$.  We thus see that every covariant
  derivative $D$ on $L$ induces by the formula above a covariant
  derivative on $TM$, and thus on $P$, which is clearly torsion-free
  and satisfies $Dc=0$.  \r

A Weyl structure $D$ is called {\em closed} (resp. {\em exact}) if $L$
carries a local (resp. global) $D$-parallel section.  As Riemannian
metrics in the conformal class $c$ correspond to positive sections of
$L$, it follows immediately that $D$ is closed (resp. exact) if and
only if $D$ is locally (resp. globally) the Levi-Civita connection of
a metric $g\in c$.

If $D$ and $D'$ are covariant derivatives on $L$, their difference is
determined by a 1-form $\tau$: $D_Xl-D'_Xl=\tau(X)l$ for all $X\in TM$
and sections $l$ of $L$. From (\ref{Koszul}) we easily obtain
$$D_XY-D'_XY=\tau(X)Y+\tau(Y)X-c(X,Y)\tau,$$ for all vector fields
$X,Y$. Here we note that the last term on the right hand side, which
is a section of $L^2\otimes \Lambda ^1M$, is identified with a vector
field using the conformal structure.

For every $X\in TM$ we define the endomorphism $\tilde\tau_X$ on $TM$
and on $L$ by \beq\label{tilde}
\tilde\tau_X(Y):=\tau(X)Y+\tau(Y)X-c(X,Y)\tau,\qquad
\tilde\tau_X(l):=\tau(X)l, \eeq and extend it as a derivation to all
weighted tensor bundles (in particular $\tilde\tau_X$ is the scalar
multiplication by $k\tau(X)$ on $L^k$). We then have \beq\label{dd}
D_X-D'_X=\tilde\tau_X, \eeq on all weighted bundles.

Consider now a metric $g$ in the conformal class $c$, or equivalently,
a positive section $l$ of $L$ trivializing $L$.  Let $D$ be a Weyl
structure on $(M,c)$ and let $\theta\in\Omega ^1(M,\RM)$ be the
connection form of $D$ on $L$ with respect to the gauge $l$:
\begin{equation}\label{cg1}
D_Xl=\theta(X)l,\qquad\forall\ X\in TM.
\end{equation} 
The 1-form $\theta$ is called the {\em Lee form} of $D$ with respect
to $g$. The curvature of $D$ on $L$ is the two-form $F:=d\theta$
called the {\em Faraday form}.

Let $R^D$ denote the curvature tensor of a Weyl structure $D$, defined
as usual for vector fields $X$, $Y$ and $Z$ by
$R^D_{X,Y}Z=[D_X,D_Y]Z-D_{[X,Y]}Z$. We also view $R^D$ as a section of
$T^*M^{\otimes 4}\otimes L^{2}$ by the formula
$R^D(X,Y,Z,T)=c(R^D_{X,Y}Z,T)$.  In contrast to the Riemannian case,
$R^D$ is not symmetric by pairs, and a straightforward calculation
shows that the symmetry failure is measured by the Faraday form $F$ of
$D$: \be\label{fa}\ba{ccl}R^D(X,Y,Z,T)-R^D(Z,T,X,Y)&=& (F(X)\wedge
Y-F(Y)\wedge X)(Z,T)\\ &&+F(X,Y)c(Z,T)-F(Z,T)c(X,Y),\ea\ee where, for
a (weighted) endomorphism $A\in \End(TM)\otimes L^k$, and
vectors $X,Y,Z,T$:
$$(A(X)\wedge Y)(Z,T):=c(A(X),Z)c(Y,T)-c(A(X),T)c(Y,Z).$$

\section{Conformal submersions}

If $(M^m,c)$ and $(N^n,c')$ are conformal manifolds, a {\em conformal
map} is a smooth map $f:M\to N$ such that
$$df|_{(\ker df)^\perp}:(\ker df)^\perp\ra df(T_xM)\subset T_{f(x)}N$$
is a conformal isomorphism for every $x\in M$. A conformal map which
is a submersion is called a conformal submersion.

\begin{elem}\label{cs1}
Let $(M^m,c)$ be a conformal manifold and let $p:M\to N$ be a
submersion onto a manifold $N^n$. Then the pull-back of the weight
bundle of $N$ is canonically isomorphic to the weight bundle of $M$.
\end{elem}

\begin{proof}
Let $L'$ and $L$ denote the weight bundles of $N$ and $M$
respectively. We decompose $TM=\ker dp\oplus (\ker dp)^\perp$ into the
vertical and horizontal distributions.  We will show that
$p^*(L')^{-n}$ is canonically isomorphic to $L^{-n}$.  Every element
$(l')^{-n}$ of the fiber of $(L')^{-n}\simeq \delta N$ at $y\in N$ can
be represented by the density $(l')^{-n}:=|\e_1\wedge...\wedge \e_n|$
where $\{\e_i\}$ is some basis of $T^*N_y$. For every $x\in p^{-1}(y)$
we then associate to the element $p^*(l')^{-n}_y$ the conformal norm
of $({p^*\e_1})_x\wedge...\wedge ({p^*\e_n})_x$, which is an element
of $L^{-n}_x$. It is straightforward to check that this isomorphism
does not depend on the choice of the basis.

\r

\obs \label{32} Notice that this result only holds for $n\ge 1$ since
  we need at least one non-vanishing 1-form in order to produce a
  weight on $N$.  \eobs

\begin{elem}\label{cs2}
Let $p:M\to N$ be a submersion with connected fibers from a conformal
manifold $(M^m,c)$ onto a manifold $N^n$. Assume that the horizontal
distribution $H:=(\ker df)^\perp$ is parallel with respect to some
Weyl structure $D$. Then the pull-back to $M$ of every covariant
weighted tensor on $N$ is $D$-parallel in the vertical directions.

Conversely, a covariant weighted tensor on $M$ which is horizontal and
$D$-parallel in the vertical directions, is the pull-back of a
covariant weighted tensor on $N$.
\end{elem}

\begin{proof} 
We first show that $D_V (p^*\o)=0$ for all 1-forms $\o$ on $N$. If $W$
is a vertical vector field, $D_VW$ is again vertical, so
$0=p^*\o(D_VW)=(D_V(p^*\o)) (W)$.  Next, if $X$ is another vector
field on $N$ and $\t X$ denotes its horizontal lift, $p^*\o(\t X)$ is
constant on each fiber of $p$, so $0=V.(p^*\o(\t X))= (D_V(p^*\o)) (\t
X)+(p^*\o)(D_V\t X)$. On the other hand $D_V\t X=D_{\t X}V+[V,\t X]$
vanishes (because $D_{\tilde X}V$ and $[V,\tilde X]$ are vertical, and
$D_V\tilde X$ is horizontal), so finally $D_V (p^*\o)=0$.

Lemma \ref{cs1} shows that the pull-back of the weight bundle $L'$ of
$N$ is isomorphic to the weight bundle $L$ of $M$.  Moreover, the
calculation above shows that \be\label{dl}D_V(p^*l')=0\ee for every
section $l'$ of $L'$ and vertical vector field $V\in \ker df$. Since
the $1$-forms and the sections of $L'$ generate the whole algebra of
covariant weighted tensors on $N$, this proves the first part of the
lemma.

For the converse part, we first show that a section $l$ of $L$ which
is parallel in vertical directions is the pull-back of a section of
$L'\ra N$. Indeed, if we take any global nowhere vanishing section
$l'$ of $L'$, one can write $l=fp^*l'$ for some function $f$ which by
\eqref{dl} is constant in vertical directions, {\em i.e.} $f$ is the
pull-back of some function $f'$ on $N$, so finally $l=p^*(f'l')$.

Let now $Q:TM^{\otimes k}\ra L^r$ be a covariant weighted tensor field
on $M$ such that $Q(X_1,\dots,X_k)=0$ whenever one of the $X_i$ is
vertical, and which is $D$-parallel in the vertical directions. For
every $y\in N$ and $x\in p^{-1}(y)\subset M$ we define
\be\label{q}\bar Q_y(Y_1,\dots,Y_k):=Q_x(\tilde Y_1,\dots,\tilde Y_k),
\ee where $Y_1,\dots,Y_k\in T_yN$ are arbitrary vectors, and $\tilde
Y_1,\dots,\tilde Y_k\in T_xM$ are their horizontal lifts. This
definition makes sense thanks to the identification of $p^*L'\simeq L$
of Lemma \ref{cs1}. In order to show that it is independent of the
choice of $x\in p^{-1}(y)$, we need to check that the right hand side
is a weight which is $D$-parallel in vertical directions. This follows
from the relations $D_V\tilde Y_i=0$, proved above, and the hypothesis
$D_VQ=0$.

\r

\begin{ecor}\label{cs3}
Under the hypothesis of the previous lemma, there exists a unique
conformal structure on $N$ turning $p$ into a conformal submersion.
\end{ecor}

\begin{proof} 
Let $c_1$ denote the restriction of the conformal structure on $M$ to
the horizontal distribution. Since the horizontal distribution is
$D$-parallel, the same holds for $c_1$. Lemma \ref{cs2} thus shows
that $c_1$ is the pull-back of some weighted tensor $c'$ on $N$, which
is clearly a conformal structure on $N$.

\r

\obs We can extend now the result of the Lemma \ref{cs2} to {\em any}
weighted tensor on $N$, respectively on $M$, because on a conformal
manifold every tensor can be seen as a covariant one (with the
appropriate weight).  \eobs

\section{Conformal products}

Let $(M_1,c_1)$ and $(M_2,c_2)$ be two conformal manifolds,
$M=M_1\times M_2$ and let $p_i:M\to M_i$ be the canonical submersions.

\begin{defi}\label{pc1}
A conformal structure on the manifold $M:=M_1\times M_2$ is said to be
{\em a conformal product} of $(M_1,c_1)$ and $(M_2,c_2)$ if and only
if the canonical submersions $p_1:M\ra M_1$ and $p_2:M\ra M_2$ are
{\em orthogonal} conformal submersions.
\end{defi}

For later use, we describe the construction of a conformal product
structure in terms of weight bundles:

\begin{prop}\label{constr} Given two conformal manifolds $(M_1^{n_1},c_1)$,
  resp. $(M_2^{n_2},c_2)$, there is a one-to-one correspondence
  between the set of {\em conformal product} structures on
  $M:=M_1\times M_2$ and the set of pairs of bundle homomorphisms
  $P_1:L\ra L_1$ and $P_2:L\ra L_2$, whose restrictions to each fiber
  are isomorphisms, such that the following diagram is commutative
  (here $L,L_1,L_2$ denote the weight bundles of $M,\ M_1$, and $M_2$
  respectively):
\begin{equation}\label{semiromb}
\begin{diagram}[size=20pt,nohug]
   & & & & L & & & & \\ & & &\ldTo^{P_1}(4,4)&\dTo&\rdTo^{P_2}(4,4)& &
   & \\ & & & & M & & & & \\ & & & \ldTo^{p_1} & & \rdTo^{p_2} & & &
   \\ L_1&\rTo&M_1& & & &M_2&\lTo&L_2 \\
\end{diagram}
\end{equation}
\end{prop}

\begin{proof} It is a general fact that for a {\em conformal map}
$f:(M,c)\ra (N,c')$ between two conformal manifolds, there is a
canonically associated bundle map $f^L:L_M\to \ L_N$ of the weight
bundles, isomorphic on each fiber (take any non-zero vector $X$ in
$(\ker df)^\perp$ and define $f^L(\sqrt{c(X,X)})=
\sqrt{c'(f_*X,f_*X)}$).

Therefore, given a conformal product structure on $M\simeq M_1\times
M_2$ ({\em i.e.}, a pair of conformal submersions $p_1:M\ra M_1$,
resp. $p_2:M\ra M_2$), we associate to it the induced bundle
homomorphisms $P_i:=p_i^L$, such that the diagram (\ref{semiromb})
commutes.

Conversely, let $M=M_1\times M_2$ and let $L$ denote the {\em weight
bundle} of the product manifold $M$ (which does not have any conformal
structure yet). Let $P_i:L\ra L_i$ be line bundle homomorphisms making
the diagram (\ref{semiromb}) commutative. The condition that $P_i$ are
isomorphic on each fiber just means that the pull-back bundles $p_i
^*L_i$ are both isomorphic with $L$. We then define $\tilde c_i\in
\Sym^2(M)\otimes p_i ^*(L_i^2)\cong \Sym^2(M)\otimes L^2$ by $\tilde
c_i(X,Y):=p_i ^*(c_i((p_i)_*(X),(p_i)_*(Y)))$ and $c=\tilde c_1+\tilde
c_2$.

\end {proof}

The following theorem, which is the main result of this section,
establishes the existence of a unique adapted Weyl structure on a
conformal product:

\begin{thrm} \label{red} 
A conformal structure $c$ on a manifold $M$ is a (local) conformal
product structure if and only if it carries a Weyl structure with
reducible holonomy. Moreover, the correspondence between the conformal
product structures and the Weyl structures with reduced holonomy is
one-to-one.
\end{thrm}
\begin{proof}
Let us first prove that if a conformal manifold $(M,c)$ carries a Weyl
structure $D$ with reduced holonomy, then it is locally a conformal
product. Let $TM=H_1\oplus H_2$ be a $D$-invariant splitting of the
tangent bundle.  Because $D$ is torsion-free, the distributions $H_1$
and $H_2$ are integrable, therefore we have two orthogonal foliations
on $M$ tangent to these distributions. Locally, $M$ is then a product
manifold $M_1\times M_2$, and the foliations above are the fibers of
the canonical projections $p_i:M\ra M_i$.  Corollary \ref{cs3} then
shows that there exist conformal structures $c_1,\ c_2$ on $M_1$,
resp. $M_2$, such that the canonical projections are conformal
submersions.  \medskip

Conversely, suppose $(M,c)$ is the a conformal product with factors
$(M_1,c_1)$, resp $(M_2,c_2)$. We look for a Weyl structure $D$ that
preserves the canonical splitting $TM=H_1\oplus H_2$, with $H_1:=\ker
dp_2$ and $H_2:=\ker dp_1$.

By Theorem \ref{ws}, the set of Weyl structures is in 1--1
correspondence with the set of connections on the weight
bundle. Therefore, it is enough to specify the corresponding
connection $D$ on $L$. We describe $D$ using the diagram
(\ref{semiromb}) as follows: the horizontal space $H_l$ of $D$ at
$l\in L$ is the direct sum of $\ker dP_1$ and $\ker dP_2$. This
horizontal space defines a linear connection because the maps $P_1,\
P_2$ commute with the scalar multiplication on the fibers. In terms of
covariant derivative, this definition amounts to say that the
pull-back of a section of $L_1$ (resp. $L_2$) is $D$-parallel in the
direction of $H_2$ (resp. $H_1$).  We need to show that the induced
Weyl structure $D$ preserves $H_1$ and $H_2$.

Since the r\^oles of $H_1$ and $H_2$ are symmetric, it is enough to
prove that $c(D_XY,Z)=0$ for all vector fields $X,Y\in H_1$ and $Z\in
H_2$. Of course, we may assume that $X,Y$ are lifts of vector fields
on $M_1$ and $Z$ is a lift of a vector field on $M_2$. Then the
brackets $[X,Z]$ and $[Y,Z]$ vanish, because the vector fields $X$ and
$Y$, resp. $Z$ are defined on different factors of the product
$M_1\times M_2$, and the scalar products $g(X,Z)=g(Y,Z)=0$ for the
same reason.  The only {\em a priori} non-vanishing terms in the Koszul
formula (\ref{Koszul}) are thus: \be
2c(D_XY,Z)=-D_Z(c(X,Y))+c([X,Y],Z).\ee The first term vanishes by the
definition of $D$ and the second one because $[X,Y]\in H_1$ and $Z\in
H_2$.

\end{proof}

\begin{defi} The  Weyl structure defined on a conformal product
  $(M,c)$ by the result above is called the {\em adapted Weyl
  structure}.
\end{defi}
\begin{defi} 
A conformal product $(M,c)$ is called a {\em closed conformal product}
if the adapted Weyl structure is closed.
\end{defi}

It is easy to check that a conformal product is closed if and only if
the diagram \eqref{semiromb} can be completed by bundle homomorphisms
$Q_i:L_i\ra L_0$, isomorphic on each fiber (where $L_0$ is the weight
bundle of the point manifold $\bullet$), such that the resulting
diagram is commutative as well:
\begin{equation}\label{romb}
\begin{diagram}[size=20pt,nohug]
   & & & & L & & & & \\ & & &\ldTo^{P_1}(4,4)&\dTo&\rdTo^{P_2}(4,4)& &
   & \\ & & & & M & & & & \\ & & & \ldTo^{p_1} & & \rdTo^{p_2} & & &
   \\ L_1&\rTo&M_1& & & &M_2&\lTo&L_2 \\ &\rdTo_{Q_2}(4,4)&&\rdTo &
   &\ldTo &&\ldTo_{Q_1}(4,4)& \\ & & & &\bullet& & & & \\ & & & &\uTo&
   & & & \\ & & & & L_0& & & & \\
\end{diagram}
\end{equation}

\begin{elem}\label{mix}
Let $F$ be the Faraday form of the adapted Weyl structure on a
conformal product. Then $F(X,Y)=0$ if $X,Y\in H_1$ or $X,Y\in H_2$.
\end{elem}

\begin{proof} 
Let $l_1$ be a section of the weight bundle $L_1$ of $M_1$. Lemma
\ref{cs1} shows that its pull-back can be identified with a section
$l$ of $L$, and Lemma \ref{cs2} shows that $D_Xl=0$ for every $X\in
H_2$. Let $X,Y$ be sections of $H_2$. Since $H_2$ is involutive, we
have $[X,Y]\in H_2$, so
$$F(X,Y)l=D_XD_Yl-D_YD_Xl-D_{[X,Y]}l=0.$$ The vanishing of $F$ on
$H_1$ is similar.

\r

\obs One can show that, for any given distribution $E$ on a conformal
manifold $M$, there is a unique {\em adapted} Weyl structure $\nb$, in
the sense that some naturally defined tensors, depending on the
splitting of $TM\simeq E\oplus E^\perp$ have minimal covariant
derivative (see \cite{calderbank}, Prop. 3.3).
\eobs

\begin{elem}\label{rest} Let $(M_1,c_1)$ and $(M_2,c_2)$ be two
  conformal manifolds and let $c$ be a conformal product structure on
  $M=M_1\times M_2$ with adapted Weyl structure $D$. Then each slice
  $M_1\times \{y\}\simeq M_1$ carries a Weyl structure $D^y$ such that
  $p_1^*(D^y_XT)=D_{X}(p_1^*T)$ at points of the form $(x,y)$ for all
  vectors $X$ and tensor fields $T$ on $M_1$. The Weyl structure $D$
  is closed if and only if all connections $D^y$ coincide.
\end{elem}

\begin{proof} The restriction to each slice $M_1\times \{y\}$ defines
  a covariant derivative $D^y$ on $M_1$, which by Lemmas \ref{cs1} and
  \ref{cs2} preserves the conformal structure of $M_1$. In order to
  prove the last statement, it is enough to consider the case when $T$
  is a vector field on $M_1$. We consider vector fields $U_2$ on $M_2$
  and $X_1,\ Y_1,\ Z_1$ on $M_1$, and denote their canonical lifts to
  $M$ by $U$, respectively by $X,Y,Z$.  Lemma \ref{cs2} shows that
  $D_UY=0$, which together with $[U,X]=0$ yields $
  R^D(U,X,Y,Z)=c(D_UD_XY,Z).$ On the other hand, $R^D_{Y,Z}U$ is
  tangent to $M_2$, so $R^D(Y,Z,U,X)$ vanishes. Plugging these two
  relations into (\ref{fa}) yields \be\label{dy} c(D_UD_XY,Z)=(Z\wedge
  F(Y)-Y\wedge F(Z))(U,X)+F(U,X)c(Y,Z).\ee If $D$ is closed,
  ({\em i.e.} $F=0$) we thus get $D_UD_XY=0$, which just means (by Lemma
  \ref{cs2} again) that the vector field $D_XY$ is the lift to $M$ of
  a vector field on $M_1$, and thus $D^y_XY$ does not depend on $y$.

Conversely, if $D^y$ does not depend on $y$, we get $D_UD_XY=0$,
therefore $R^D(U,X)Y=0$. But the endomorphism $R^D(U,X):TM\ra TM$
decomposes as a skew-symmetric endomorphism (in Riemannian geometry
this is the only piece) and a symmetric one, which is equal to
$F(U,X)\Id$. This piece has to vanish now that $R^D(U,X)Y=0$,
hence $F(U,X)=0$. This, together with Lemma \ref{mix}, proves that
$F=0$.

\r

\obs\label{obs} It follows from Lemma \ref{mix} that $D^y$ are {\em
closed} Weyl structures on $M_1$, for any $y\in M_2$. Moreover, a
$D^y$-parallel metric on $M_1$ is only determined up to a factor
depending on $y$ alone, and the
proof of Theorem \ref{red} tells us to which proportionality class 
of sections of $L_1\ra M_1$ it corresponds: Any section
$\sigma_2$ of the scale bundle $L_2\ra M_2$ induces a section
$\tilde\sigma_2$ of $L\simeq p_2^*L_2$ over $M_1\times M_2$, which is,
by definition, $D$-parallel in the directions of the $M_1$-leaves.

In other words, any metric $g_2\in c_2$ on $M_2$ defines a metric
$g\in c$, such that the restriction of $g$ to the $M_2$ leaves is
$g_2$ (leaf-independent metric), and $D^y$ is the Levi-Civita
connection of the metric $g|_{M_1\times \{y\}}$ (which depends on
$y$).

In Section 6 we give more details about this construction.

 \eobs

\obs The results in this section hold under the implicit assumption
that each factor of the conformal product has dimension at least one,
because we need to identify the weight bundle of the product with the
pull-back of those of each factor (see Remark \ref{32}).  \eobs

\section{Weyl-parallel forms}

In this section we study the following problem, which motivates, as we
shall see below, the notion of conformal product. Given a conformal
manifold $(M^n,c)$, and a weightless $k$-form $\omega\in\Lambda ^k_0M$
for some $1\le k\le n-1$, does there exist a Weyl structure $D$ such
that $\o$ is $D$-parallel?

\obs We only consider weightless forms since if $T$ is a
(non-vanishing) $D$-parallel weighted tensor, then $T/\sqrt{c(T,T)}$
is a $D$-parallel weightless tensor. Moreover, if $T$ has non-zero
weight, the Weyl structure is exact (due to the fact that the
conformal norm $c(T,T)$ is a $D$-parallel weight).\eobs

\obs For a conformal product $M=M_1\times M_2$, the weightless volume
forms on the factors induce weightless forms which are parallel with
respect to the unique adapted Weyl structure (which preserves the
splitting $TM\simeq TM_1\oplus TM_2$). This will turn out to be one of
the main examples of parallel weightless forms on conformal
manifolds.\eobs

We start with the following useful result.

\begin{elem} \label{tens1}Let $\omega\in\Lambda ^k_0M$ be a weightless
  $k$-form ($1\le k\le n-1$). Then there exists {\em at most} one Weyl
structure with respect to which $\omega$ is parallel.
\end{elem}

\begin{proof}
Assume that $D\omega=D'\omega=0$, denote $D'=D+\tilde\tau$ like in
(\ref{dd}) and let $g\in c$ be any ground metric, used to identify
1-forms and vectors. For every vector field $X$ we get \be
\label{wp}0=D_X\omega-D'_X\omega=\tilde\tau_X(\omega)
=X\wedge(\tau\i\omega)-\tau\wedge(X\i\omega).\ee Taking the exterior
product with $X$, and setting $X$ to be an element of some
$g$-orthonormal basis $\{e_i\}$ we get, after adding up all the
resulting equations, that $0=\tau\wedge(k\omega)$. Similarly, taking
the interior product with $X$ and summing over some $g$-orthonormal
basis $X=e_i$ yields $0=(n-k+1)(\tau\i\omega)-\tau\i\omega$, thus
$\tau\i\omega=0$. But, if $\tau\ne 0$, the condition
$\tau\wedge\omega=0$ implies that $\tau$ is a factor of
$\omega=\tau\wedge\omega'$, and $\tau\i\omega=0$ implies $\omega'=0$
which contradicts the non-triviality of $\omega$.

\r

\obs Consider the linear map $\alpha:T^*M\to T^*M\otimes \Lambda
^k_0M$ defined by
$$\alpha(\tau)(X) =\tilde\tau_X(\omega).$$ The proof of the lemma
above show that if $\o$ is a nowhere vanishing section of $\Lambda
^k_0M$, then $\alpha$ is injective, and there exists a unique Weyl
structure $D^\omega$ such that $D^\omega\o$ is orthogonal to the image
of $\alpha$. We call $D^\omega$ the {\em minimal} Weyl structure
associated to $\omega$.  \eobs

Our problem can thus be reformulated as follows: Given a conformal
manifold $(M^n,c)$, find all nowhere vanishing sections $\omega$ of
$\Lambda ^k_0M$ for some $1\le k\le n-1$, such that
$D^\omega\omega=0$.  Notice that $\o$ being nowhere vanishing is a
necessary condition for the existence of a Weyl structure $D$ with
$D\o=0$.

We start with the case where the minimal Weyl structure $D^\o$
associated to $\o$ is closed. Since our study is local, there exists
some metric $g\in c$ whose Levi-Civita connection is $D^\o$. Since $g$
trivializes the weight bundle, $\o$ induces a parallel $k$-form on
$M$. Using the Berger-Simons holonomy theorem, our problem in this
case reduces to a purely algebraic one and its answer can be
synthesized in the following classical statement.

\begin{ath}\label{riem} Let $(M^n,g)$ be a simply connected Riemannian
  manifold with Levi-Civita covariant derivative $\n$. The space of
  parallel $k$-forms on $M$ is isomorphic to the space of fixed points
  of the holonomy group $\Hol(\n)$ of $\n$ acting on $\L^k(\RM^n)$. If
  $\Hol(\n)$ acts irreducibly on $\RM^n$, then either $M=G/H$ is a
  symmetric space, $Hol(\n)=H$ and $H$ is listed in \cite{bes}, Tables
  1-4, pp.201-202, or $Hol(\n)$ belongs to the Berger list
  (\cite{bes}, Corollary 10.92). If $\Hol(\n)$ acts reducibly on
  $\RM^n$, then $\Hol(\n)$ is diagonally embedded in $\SO_n$ as a
  product $\Hol(\n)=H_1\times\ldots\times H_l$, where each
  $H_i\in\SO_{n_i}$ belongs to the lists above.
\end{ath}
 
Another preliminary result concerns the special case of a weightless
2-form whose associated endomorphism is an almost complex structure
$J$. This case is also classical and completely understood (see
{\em e.g.} \cite{pps}, Section 2):

\begin{elem}\label{lck} An almost complex structure $J$
  compatible with the conformal structure on $(M^{2m},c)$ is parallel
  with respect to some Weyl structure $D$ if and only if
\begin{itemize}
\item $m=1\mbox{ or }2$ and $J$ is integrable;
\item $m\ge 3$ and $(M,c,J)$ is a locally conformally K\"ahler
  manifold.
\end{itemize}
\end{elem}
\begin{proof} We provide the proof for the reader's convenience.
Assume that $DJ=0$. Since $D$ is torsion free, $J$ is integrable
(\cite{kn2}, Ch. 9, Corollary 3.5). Let $g$ be any metric in the
conformal class $c$ with the associated 2-form $\o(X,Y):=g(JX,Y)$.
Using $g$, we identify 1-forms with vectors, and 2-forms with
skew-symmetric endomorphisms. If $\tau$ denotes the Lee form of $D^J$
with respect to the Levi-Civita connection $\n$ of $g$, we have for
every tangent vector $X$
\beq\label{d1}0=D^J_X=\n_XJ+\tilde\tau_XJ.\eeq From (\ref{tilde}) we
compute \bea (\tilde\tau_XJ)(Y)&=&\tau_X(JY)-J(\tau_XY)\\
&=&\tau(X)JY+\tau(JY)X-g(X,JY)\tau-J(\tau(X)Y+\tau(Y)X-c(X,Y)\tau)\\
&=&(X\wedge J\tau+JX\wedge\tau)(Y), \eea whence \beq\label{f}
\n_X\o=-(X\wedge J\tau+JX\wedge\tau) \eeq From (\ref{d1}) we get in a
local orthonormal basis $\{e_i\}$:
$$ d\o=\sum_i
e_i\wedge\n_{e_i}\o=-\sum_ie_i\wedge(e_i\wedge\tau+Je_i\wedge\tau)=
-2\o\wedge\tau.$$ Taking the exterior derivative in this last relation
yields
$$0=d^2\o=-2d\o\wedge\tau-2\omega\wedge d\tau=-2\omega\wedge d\tau.$$
For $m\ge 3$ the exterior product with $\o$ is injective, so
$d\tau=0$, and $(M,c,J)$ is thus locally conformally K\"ahler.

Conversely, assume first that $m\ge 3$ and $(M,c,J)$ is locally
  conformally K\"ahler. If $g_1$ and $g_2$ are local K\"ahler metrics
  on open sets $U_1$ and $U_2$ in the conformal class $c$, then $g_1$
  and $g_2$ are homothetic on $U_1\cap U_2$. The Levi-Civita
  connections of all such metrics thus define a global Weyl structure
  $D$ leaving $J$ parallel.

If $m=2$ and $J$ is integrable, let $g$ be any metric in the conformal
class $c$ with the associated 2-form $\o(X,Y):=g(JX,Y)$. Since the
wedge product with $\o$ defines an isomorphism $\L^1M\cong \L^3M$,
there exists a unique 1-form $\tau$ such that
$d\o=-2\o\wedge\tau$. From \cite{kn2}, Ch. 9, Proposition 4.2 we
obtain
\bea\n_X\o(Y,Z)&=&-(\o\wedge\tau)(X,JY,JZ)+(\o\wedge\tau)(X,Y,Z)\\
&=&-(JX\wedge\tau)(JY,JZ)+(JX\wedge\tau)(Y,Z)\\&=&(X\wedge
J\tau+JX\wedge\tau)(Y,Z) \eea (notice that in the definition of $d\o$
there is an extra factor 3 with the conventions in \cite{kn2}). By
(\ref{f}), this means that $J$ is parallel with respect to the Weyl
structure $\n+\tilde\tau$.

\r

Before stating our main result, we need one more preliminary statement
concerning conformal products.

\begin{prop}\label{restr}
A conformal product $M=M_1\times M_2$, with $\dim M_i=n_i\ge 1$,
admitting a non-trivial weightless form $\o\in
C^\infty\left(p_1^*(\L_0^k M_1)\right)$, $1\le k\le n_1-1$, which is
parallel with respect to the adapted Weyl structure, is a closed
conformal product.
\end{prop} 
In other words, on a non-closed conformal product, the only weightless
forms of {\em pure type} $M_1$ or $M_2$ are the volume forms of the
factors.
\begin{proof} 
Let $M=M_1\times M_2$ and $TM=H_1\oplus H_2$ be the corresponding
orthogonal splitting of the tangent space $H_i\simeq p_i^*TM_i$, where
$p_i:M\ra M_i$ are the canonical projections. Let $D$ be the adapted
Weyl structure on $M$.
Since $D_X\o=0,\ \forall\; X\in H_2$, Lemma \ref{cs2} shows that $\o$
is the pull-back of a weightless form
$\o_1\in\C^\infty\left(\L^k_0M_1\right)$, hence
$$\o=p_1^*\o_1.$$

From Lemma \ref{rest}, $D$ induces a Weyl structure $D^y$ on each
slice $M_1\times \{y\}$ and the equation $D_Y \o=0, \ \forall\; Y\!
\in H_1,$ shows that the restriction of $\o$ to each slice $M_1\times
\{y\}$ is $D^y$-parallel. Of course, all these restrictions coincide
with $\o_1$, so we have
$$D^y_X\o_1=0,\ \forall X\in TM_1\mbox{ and }\forall y\in M_2.$$ Now,
as the degree of $\o_1$ lies between $1$ and $n_1-1$, the equation
above implies (using Lemma \ref{tens1}) that all Weyl structures $D^y$
coincide, so by Lemma \ref{rest} again, the Weyl structure $D$ is
closed and $(M,c)$ is thus a closed conformal product.

\r

We are now ready for the classification of conformal manifolds
carrying conformally parallel forms.

\begin{ath}\label{main}
Let $(M^n,c)$ be a conformal manifold and let $\o\in{\mathcal
C}^\infty(\L^k_0M)$ be a weightless $k$-form ($1\le k \le n-1$) such
that there exists a Weyl structure $D$ with respect to which $\o$ is
parallel. Then the following (non-exclusive) possibilities occur:
\begin{enumerate}
\item $D$ is closed, so Theorem \ref{riem} applies.
\item $M$ has dimension $4$, $k=2$, the endomorphism of $TM$
corresponding to $\o$ is, up to a constant factor, an integrable
complex structure, and $D$ is its canonical Weyl structure.
\item $(M^n,c)$ is a conformal product of $(M^{n_1},c_1)$ and
$(M^{n_2},c_2)$, $D$ is the adapted Weyl structure, and
\begin{enumerate}
\item if $n_1\ne n_2$, $\o=\l \o_i$ for some $\l\in\RM$, where $\o_i$
denotes the weightless volume form of the factor $M_i$;
\item if $n_1=n_2$, then $\o=\l \o_1+\mu \o_2$ for some
$\l,\mu\in\RM$.
\end{enumerate}\end{enumerate}
\end{ath}

\begin{proof} Let us first consider the case of forms of low
degree. If $\o$ is a $D$-parallel weightless 1-form, its kernel
defines a $D$-parallel distribution of codimension 1, so by Theorem
\ref{red} $(M,c)$ is a conformal product where one factor is
one-dimensional and $\o$ is its weightless volume form (case 3a).

The case when $\o$ has degree 2 has a special geometrical meaning,
since it can be seen as a skew-symmetric endomorphism $J$ of $TM$.  As
such, its square is a parallel symmetric endomorphism, therefore its
eigenvalues are constant (and non-positive) and the corresponding
eigenspaces are parallel. There are two cases to be considered.

If $J^2$ has only one eigenvalue, one may assume after rescaling that
$J^2=-\Id_{TM}$, so by Lemma \ref{lck}, either $n=4$ and we are in
case 2, or $n\ge 6$, $(M,c)$ is locally conformally K\"ahler, and $D$
is closed (case 1).

If $J^2$ has at least two eigenvalues, we denote by $H_1$ one of the
eigenspaces, and by $H_2$ its orthogonal complement in $TM$. The
splitting $TM=H_1\oplus H_2$ is thus $D$-parallel, so $M$ has to be a
non-trivial conformal product by Theorem \ref{red}. On the other hand,
$J$ splits into $J=J_1+J_2$, where $J_1,J_2$ are the restrictions of
$J$ to $H_1$, resp. $H_2$. The form $\o$ splits accordingly into
$\o=\o_1+\o_2$, with $\o_i\in{\mathcal
C}^\infty\left(p^*_i(\L^2_0M_i)\right)$. By Proposition \ref{restr},
either the conformal product $M_1\times M_2$ is closed (case 1), or
the non-trivial $\o_i$ is a pull-back of a weightless volume form on
$M_i$. Therefore, if $D$ is non-closed, $J^2$ has exactly two
eigenvalues, which are either both non-zero (then $n=4$ and we are in
case 3b) or only one is non-zero, and we are in case 3a. Note that in
the latter case, $\o$ is defined as the pull-back of a volume form on
a 2-dimensional conformal factor, thus it is decomposable.

In order to proceed, we make use of Merkulov-Schwachh\"ofer's
classification of torsion-free connections with irreducible holonomy
\cite{SM}. Their result, in the particular case of Weyl structures,
states that there are four possibilities: Either $n=4$, or $D$ has
full holonomy $CO^+(n)$, or $D$ is closed, or $D$ has reducible
holonomy.

If $n=4$, the case where the degree of $\o$ is 1 or 2 has already been
considered, and if $\o$ has degree 3, its Hodge dual is again a
$D$-parallel weightless 1-form, so we are in case 3a.

If $D$ has full holonomy $CO^+(n)$, there is of course no $D$-parallel
weightless $k$-form on $M$ for $1\le k \le n-1$.

If $D$ is closed we are already in case 1.

For the rest of the proof, we thus may assume that the holonomy of $D$
acts reducibly on $TM$ and $\dim M\ge 5$.  By Theorem \ref{red},
$(M^n,c)$ is locally a conformal product of $(M^{n_1},c_1)$ and
$(M^{n_2},c_2)$ and $D$ is the adapted Weyl structure. From Lemma
\ref{cs2} we have the following $D$-parallel decomposition:
\beq\label{dec}\L^k_0M\simeq \bigoplus_{k_1+k_2=k}
p_1^*(\L^{k_1}_0M_1)\otimes p_2^*(\L^{k_2}_0M_2).\eeq If for every
$k_1\ge 1,\ k_2\ge 1$ the components of $\o$ in
$p_1^*(\L^{k_1}_0M_1)\otimes p_2^*(\L^{k_2}_0M_2)$ vanish, then
$\o=\o_1+\o_2$, with $\o_i\in C^\infty\left(p_i^*(\L^k_0M_i)\right)$,
and $\o_1,\o_2$ are both parallel. Proposition \ref{restr} then
implies that either $D$ is closed (case 1) or $\o_i$ are both
pull-backs of weightless volume forms on the factors (not both
trivial). But this can only happen if $k$ is equal to one of the
dimensions $n_1,n_2$ (case 3a) or to both of them (case 3b).

To deal with the cases when $\omega$ is not a (combination of) pure
type form, we may assume without loss of generality that $\o$ is a
non-trivial section of $p_1^*(\L^{k_1}_0M_1)\otimes
p_2^*(\L^{k_2}_0M_2)$ with $k_1\ge 1,\ k_2\ge 1$.  By considering the
Hodge dual of $\o$ (which is $D$-parallel as well), we may even assume
$k_1<n_1$ and $k_2<n_2$.  We will show that, in this case, $D$ must be
closed (case 1).

Let $R^D$ denote the curvature tensor of $D$. If $X\in TM_1$ and $A\in
TM_2$, we have $R^D(\.,\.,X,A)=0$. Since $D\o=0$ we also have
$R^D(X,A)(\o)=0$.  From (\ref{fa}) we thus get
$$R^D(X,A,\cdot,\cdot)=(F(X)\wedge A-F(A)\wedge X)(\cdot,\cdot)
+F(X,A)c(\cdot,\cdot),$$ and note that the first part is a
skew-symmetric endomorphism of $TM$, and the second term is a multiple
of the identity. That last one acts trivially on weightless forms, so
we are left with four terms in $R^D(X,A)(\o)$ and get: \be\label{ff}
F(A)\wedge(X\i\o)-X\wedge(F(A)\i\o)-
F(X)\wedge(A\i\o)+A\wedge(F(X)\i\o)=0, \ee for all $X\in TM_1$ and
$A\in TM_2$.

We now choose a metric $g\in c$ and identify the weightless form $\o$
with the corresponding $p$-form of constant $g$-length. If $\theta$
denotes the Lee form of $D$ with respect to $g$ and $\n$ the
Levi-Civita covariant derivative of $(M,g)$, we have \be
\n_U\o=-\tilde\theta_U\o=\theta\wedge(U\i\o)- U\wedge(\theta\i\o).
\ee By contraction we easily get $d\o=-p\theta\wedge\o$. Taking the
exterior derivative in this relation yields \be \label{th}
d\theta\wedge\o=0.  \ee Since $d\theta=F$, taking the interior product
with some vector in (\ref{th}) yields \be\label{u} F\wedge
(U\i\o)=-F(U)\wedge\o,\qquad\forall\, U\in TM.  \ee

The rest of the proof is purely algebraic. Let $X_i$, $A_j$ be local
orthonormal basis of $H_1:=TM_1$ and $H_2:=TM_2$. By Lemma \ref{mix}
we have
$$F=\sum_{i,j}f_{ij} X_i\wedge A_j,$$ therefore $F=\sum X_i\wedge
F(X_i)=\sum A_j\wedge F(A_j)$.  We also introduce the notations
$\phi=\sum F(X_i)\wedge (X_i\i\o)= -\sum A_j\wedge (F(A_j)\i\o)$ and
$\psi=\sum F(A_j)\wedge (A_j\i\o)= -\sum X_i\wedge (F(X_i)\i\o)$.

For every $\alpha\in p_1^*(\L^{p}_0M_1)\otimes p_2^*(\L^{q}_0M_2)$ we
have $\sum X_i\wedge (X_i\i\alpha)=p\alpha$ and $\sum A_j\wedge
(A_j\i\alpha)=q\alpha$. Taking the wedge product with $X$ in
(\ref{ff}), summing over $X=X_i$ and using (\ref{u}) yields
\be\label{ff1} 0=-k_1 F(A)\wedge\o- F\wedge(A\i\o)+A\wedge\psi=(1-k_1)
F(A)\wedge\o+A\wedge\psi.  \ee One last contraction with $A_i$ in
(\ref{ff1}) gives $(k_1+n_2-k_2)\psi=0$ (note that $\psi\in
p_1^*(\L^{k_1+1}_0M_1)\otimes p_2^*(\L^{k_2-1}_0M_2)$). Plugging back
into (\ref{ff1}) yields \be\label{e1} (1-k_1)F(A)\wedge\o=0,\qquad
\forall\, A\in TM_2.  \ee In a similar way one obtains \be\label{e2}
(1-k_2)F(X)\wedge\o=0,\qquad \forall\, X\in TM_1, \ee and, by
replacing $\o$ with its Hodge dual, and taking the Hodge dual of the
equations for $*\o$ analogous to (\ref{e1}) and (\ref{e2}), we also
get \be\label{e3} (n_1-k_1-1)F(A)\i\o=0,\qquad \forall\, A\in TM_2,
\ee \be\label{e4} (n_2-k_2-1)F(X)\i\o=0,\qquad \forall\, X\in TM_1.
\ee

If $2\le k_1\le n_1-2$, equations \eqref{e1} and \eqref{e3} show that
$F=0$.  Similarly, if $2\le k_2\le n_2-2$, equations \eqref{e2} and
\eqref{e4} show that $F=0$. We are thus left with four cases:
\beq\label{cases}(k_1,k_2)\in\{(1,1),(1,n_2-1),(n_1-1,1),(n_1-1,n_2-1)\}.\eeq
Now, the pull-back to $M$ of the Hodge operator of $M_1$ defines an
operator on $\L^*_0M$, which maps each component
$p_1^*(\L^{k_1}_0M_1)\otimes p_2^*(\L^{k_2}_0M_2)$ in the
decomposition \eqref{dec} isomorphically onto
$p_1^*(\L^{n_1-k_1}_0M_1)\otimes p_2^*(\L^{k_2}_0M_2)$ by
$$*_1[p_1^*(\o_1)\wedge p_2^*(\o_2)]:=p_1^*(*\o_1)\wedge
p_2^*(\o_2).$$ One defines $*_2$ in a similar way and Lemma \ref{cs2}
shows that $*_1$ and $*_2$ are $D$-parallel. Using these ``partial
Hodge operators", it suffices to study only the first case in
\eqref{cases}, {\em i.e.} we can assume that $\o$ is a $D$-parallel 2-form
in $p_1^*(\L^{1}_0M_1)\otimes p_2^*(\L^{1}_0M_2)$.
 
As $n\ge 5$, by looking at the case of $D$-parallel 2-forms treated
above, we see that $D$ is closed unless $\o$ is decomposable. But this
would imply that $\o=\eta_1\wedge \eta_2$, where each of the
weightless 1-forms $\eta_i$ define a $D$-parallel distribution
included in $H_i$. The forms $\eta_i$ are thus $D$-parallel and Lemma
\ref{restr} implies that $D$ is closed.

\end{proof}

Looking back to the proof of Theorem \ref{main}, we see how different
the non-closed case is from the case of a closed Weyl structure, and
this despite the fact that the results are essentially similar: as in
the classical Riemannian case, a weightless form which is parallel for
a Weyl structure either defines a special, irreducible, holonomy, or
the manifold is locally a product and the form is a linear combination
of pull-backs of the volume forms of the factors (in the Riemannian
case, other pull-backs may occur if the factors have reduced
holonomy).

But while in Riemannian geometry the richer case is the one with
irreducible, non-generic, holonomy -- and these {\em special
geometries}, despite extensive research in the last decades, are far
from being completely understood --, in non-closed Weyl geometry this
situation occurs only in dimension 4, and there it defines a rather
simple structure. It appears that for {\em non-closed} Weyl structures
it is the case with {\em reduced} holonomy which is more interesting,
and the reason is that the holonomy group -- although defining a local
product structure on the manifold -- is not itself a product like in
the Riemannian situation.

The following consequence of the Theorem \ref{main} sheds some light
on the reduced holonomy group of a non-closed Weyl structure:

\begin{ecor}\label{holo}
Let $(M_1^{n_1},c_1)$ and $(M_2^{n_2},c_2)$ be two conformal manifolds
and let $c$ be a non-closed conformal product structure on
$M=M_1\times M_2$ with adapted Weyl structure $D$. If the dimension
$n$ is neither $2$ nor $4$, then the only $D$-parallel 
distributions on $M$ are 
the kernels $H_2$ and $H_1$ of the canonical projections of $M$ on
$M_1$, respectively $M_2$.
\end{ecor}

\begin{proof}
Let $\o_i$ denote the ($D$-parallel) weightless volume form of $H_i$.
If $H$ is a $D$-parallel distribution, its weightless volume form $\o$
is $D$-parallel. Consider first the case $n=3$: We can assume
$\o_1$ and $\o$ are weightless 1-forms, and suppose they are not
proportional (otherwise $H=H_1$). As the restriction of $\o$ to $H_1$
is $D$-parallel as well, we can assume that there is a $D$-parallel
orthogonal basis of 1-forms on $M$. We denote by $L_i$, $i=1,2,3$, the
three $D$-parallel, mutually orthogonal line distributions. By
considering all three possible decompositions of $TM$ as $L_i\oplus
L_i^\perp$, we see that Lemma \ref{mix} implies that $F=0$, hence $D$
is closed. 

Suppose now that $D$ is non-closed and $\dim(M)> 4$. The proof of Theorem
\ref{main} shows that $\o$ is either proportional to $\o_1$ or $\o_2$
(in which case $H$ is equal to $H_1$ or $H_2$), or $n_1=n_2$ and
$\o=\l\o_1+\mu\o_2$ for some $\l,\mu\in\RM^*$. We claim that this latter
case is impossible. Let $X$ be some vector field in $H$ whose
projections $X_i$ onto $H_i$ are both non-vanishing (such a vector
field exists locally because $H$ is not equal to $H_1$ or $H_2$), and
let $\s=c(X,.)$ be the dual 1-form of weight $1$, which decomposes
correspondingly as $\s=\s_1+\s_2$. We clearly have $\s\wedge\o=0$,
whereas
$$\s\wedge(\l\o_1+\mu\o_2)=\l\s_2\wedge\o_1+\mu\s_1\wedge\o_2$$ and
the two terms on the right hand side are non-vanishing and have
bi-degree $(n_1,1)$ and $(1,n_2)$ with respect to the decomposition
\eqref{dec}. The assumption $n_1+n_2\ne 2$ shows that their sum can not
vanish, a contradiction which proves our claim.

\r

\obs The proposition above establishes a 1--1 correspondence between
the non-closed Weyl structures $D$ on $M$ with reducible holonomy and
conformal product structures, defined by a pair of complementary
$D$-invariant distributions, provided $\dim M\ne 2,4$. Note that the
restriction on the dimension of $M$ is necessary. Indeed, in 
dimension 2, any conformal product of two 1-dimensional manifolds
admits parallel lines (distributions of rank 1) in any directions:
Simply consider linear combinations of the two volume forms on the
factors. But not all such conformal products are closed. 

In dimension 4, there exist local examples of non-closed conformal
{\em multi-products}, {\em i.e.}, conformal manifolds admitting a non-closed
Weyl structure that leaves invariant more than one pair of orthogonal
distributions. These examples are described in Proposition \ref{bh}
below.\eobs

\section{Examples and applications to Einstein-Weyl geometry}

\subsection{Curvature of a conformal product}

Let $(M_1,g_1)$ and $(M_2,g_2)$ be Riemannian manifolds of dimensions
$n_1$, resp. $n_2$, with $n:=n_1+n_2$, and let $f_1,f_2:M\ra\R$ be
$C^\infty$ functions, where $M:=M_1\times M_2$. Consider the following
metric on $M$: 
$$g:=e^{f_1}g_1+e^{f_2}g_2.$$
Then $(M,[g])$ is a conformal product of the conformal manifolds
$(M_1,[g_1])$ and $(M_2,[g_2])$, in fact, any conformal product is
locally of this form.  
Moreover, for two couples $(f_1,f_2)$,
resp. $(f_1',f_2')$, the resulting conformal structures are equal if
and only if $f_1-f_2=f'_1-f'_2$. 

$(M,c)$ is a closed conformal product if and only if it
can be locally expressed in the form above, each $f_i$ being a
pull-back of a function on $M_i$.

From Remark \ref{obs} (see also Lemma \ref{cs2}), the adapted Weyl
structure $D$ of the conformal product $M=M_1\times M_2$ satisfies the
following: The metric $e^{f_1-f_2}g_1+g_2$ is $D$-parallel in the
$H_1$-directions, and $g_1+e^{f_2-f_1}g_2$ is $D$-parallel in the
$H_2$-directions. This implies (in the sequel, $X_i,Y_i,Z_i$ are
vector 
fields on $M_i$, equally considered as vector fields on the
corresponding leaves on $M$): 
\bi
\item $D_{X_1}X_2=D_{X_2}X_1=0$.
\item
For $i=1,2$, $D_{X_i}Y_i$ coincides with the Levi-Civita connection of
the metric $e^{\varepsilon(i)(f_1-f_2)}g_i$, where $\varepsilon(1):=1$
and $\varepsilon(2):=-1$.
\item the Lee form of $D$ with respect to $g$ is given by
$\theta(X)=-X_1(f_2)-X_2(f_1)$, where $X_i$ denote the components of
$X$ with respect to the decomposition $TM=TM_1\oplus TM_2$.
\ei

As before, we denote by $F$ the Faraday form of $D$, which by Lemma
\ref{mix} satisfies
$$F(X_i,Y_i)=0,\qquad \forall X_i,Y_i\in TM_i,$$
so it is obtained by extending a section $F_0$ of $H_1^*\otimes H_2^*$
to a skew-symmetric bilinear form on $TM=H_1\oplus H_2$. 

$F_0$ can also be extended to a {\em symmetric} bilinear form $\hat F$
on $TM$:  
$$\hat F(X_i,Y_i):=0;\ \hat F(X_1,X_2):= F(X_1,X_2);\ \hat
F(X_2,X_1):=-F(X_2,X_1),\ \forall X_i,Y_i\in TM_i.$$  

\begin{elem} Let $M:=M_1\times M_2$ be endowed with a conformal
  product structure $c:=[e^{f_1}g_1+e^{f_2}g_2].$ For $i=1,2$ denote by
  $\Ric^i$ the Ricci tensor of the metric
  $e^{\varepsilon(i)(f_1-f_2)}g_i$, viewed as a symmetric bilinear
  form on $TM_i$. Then the Ricci tensor of the adapted
  Weyl structure $D$ is given by 
\beq\label{ric}
\Ric^D=\Ric^1\oplus \Ric^2 + \frac{2-n}{2}F+\frac{n_1-n_2}{2}\hat F.
\eeq
\end{elem}
\begin{proof} Recall that the Ricci tensor of a Weyl structure $D$ 
is defined by
\beq\label{ric1}\Ric^D(X,Y):=\frac12\sum _{k=1}^n
(g(R^D_{X,e_k}e_k,Y)-g(R^D_{X,e_k}Y,e_k)),\eeq
where $g$ is an arbitrary metric in the conformal class and $\{e_k\}$ is a
local $g$-orthonormal frame (cf. \cite{g}, where, however, a
different sign convention for the curvature tensor is used). It is
straightforward to show that the skew-symmetric part of $\Ric^D$
equals $\frac{2-n}{2}F$ (this is actually an easy consequence of
\eqref{fa} and holds for any Weyl
structure). If $X,Y$ are vector fields on $M_1$, each summand in
\eqref{ric1} vanishes for $k\ge n_1+1$ since $R^D_{Z,T}$ preserves the
splitting $TM=TM_1\oplus TM_2$ for all $Z,T$. Moreover, $R^D$ is just
the Riemannian curvature of the metric $e^{f_1-f_2}g_1$ on vectors tangent to
$M_1$. This shows that $\Ric^D(X,Y)=\Ric^1(X,Y)$ whenever $X$ and $Y$
are tangent to $M_1$ (and similarly $\Ric^D(X,Y)=\Ric^2(X,Y)$ for $X$, $Y$
tangent to $M_2$). Finally, using \eqref{fa} we easily obtain
$\Ric^D(X,Y)=(1-n_2)F(X,Y)$ and $\Ric^D(Y,X)=(1-n_1)F(Y,X)$ for $X\in
TM_1$ and $Y\in TM_2$, which in turn implies \eqref{ric}.

\r

\subsection{Einstein-Weyl conformal products}

A Weyl manifold $(M,c,D)$ is called Ein\-stein-Weyl if the trace-free
{\em symmetric} part of the Ricci tensor $\Ric^D$ vanishes.
In this subsection we will give the local characterization of all
Einstein-Weyl structures $(M,c,D)$ in dimension 4 with reducible
holonomy. Of course, we will be mainly interested in {\em non-closed}
Weyl structures, the closed case being locally Riemannian, thus
well-understood. Note that the {\em scalar curvature} of a non-closed
Weyl structure, defined as usual as the trace of the Ricci tensor, is
a section in a weight bundle and therefore never (covariantly)
constant, unless it vanishes identically. The {\em scalar-flat}
Einstein-Weyl structures are thus a special class of Weyl structures,
which in our case provide examples of conformal products admitting
multiple reductions (see previous section and Proposition \ref{bh}). 

\begin{epr} \label{ans} 
A non-closed Weyl manifold $(M,c,D)$ of dimension $4$, with reduced
holonomy, is Einstein-Weyl if and only if it is locally isomorphic to
a conformal product $M_1\times M_2$, $c=[g_1+e ^{2f}g_2]$, where $M_1$
  and $M_2$ are open sets of $\RM^2$, $g_i$ is the flat metric on
  $M_i$ and the function $f:M_1\times M_2\subset \R^4\to \R$ satisfies
  the Toda-type equation
\beq\label{toda} e
^{2f}(\partial_{11}f+\partial_{22}f)+\partial_{33}f+\partial_{44}f=0.
\eeq
Moreover, the Einstein-Weyl structure is scalar flat if and only if $f$
is bi-harmonic ({\em i.e.} the restrictions of $f$ to $M_1\times \{x_2\}$ and
  to $\{x_1\}\times M_2$ are harmonic for all $x_i\in M_i$). 
\end{epr} 

\begin{proof}
By Theorem \ref{red}, $(M,c,D)$ is reducible if and only if $(M,c)$ is
a conformal product $M=M_1^{n_1}\times 
M_2^{n_2}$, $c=[g_1+e ^{2f}g_2]$, and $D$ is the adapted Weyl structure. 

The first remark is that $M_1$ and $M_2$ must have the same
dimension. Indeed, since the Faraday form is non-zero, Equation
\eqref{ric} shows that $(M,c,D)$ is Einstein-Weyl if and only if
$n_1=n_2=2$ and 
\beq\label{we4}\Ric^1+\Ric^2= \f(g_1+e ^{2f}g_2)\eeq 
for some function
$\f:M\to\R$, where we remind that $\Ric^i$ is the Ricci tensor of the
metric $e ^{2\e(i)f}g_i$. Every 2-dimensional metric being locally
conformal to the flat metric, one can assume that $g_1$ and $g_2$ are
flat. Using the basic formulas for the conformal change of the Ricci
tensor (\cite{bes}, 1.159d), \eqref{we4} becomes
\beq\label{sr}-(\D_1f)g_1+(\D_2f)g_2=\f(g_1+e ^{2f}g_2),\eeq
where $\D_i$ denote the partial Laplacians on $\R^4=\R^2\times \R^2$:
$$\D_1f:=\partial_{11}f+\partial_{22}f,\qquad
\D_2f:=\partial_{33}f+\partial_{44}f.$$
Equation \eqref{sr} is clearly equivalent to \eqref{toda}. Moreover,
the trace of 
$\Ric^D$ with respect to the metric $g_1+e ^{2f}g_2$ is $-2\D_1f+2e
^{-2f}\D_2 f$. This shows that $D$ is Einstein-Weyl and scalar-flat if
and only if $\D_1f=\D_2f=0$.

\r

\obs If $f:U\times V\ra \R$, $U,V\subset\C$ is {\em bi-harmonic},
{\em i.e.}, harmonic with respect to both variables $z\in U,\ w\in V$, then  
$$f=\mathrm{Re}(F)+\mathrm{Re}(\tilde G),$$
where $F,G:U\times V\ra\C$ are holomorphic, and $\tilde
G(z,w):=G(z,\bar w)$, $\forall (z,w)\in U\times V$. It turns out that
the case where one of $F,G$ is trivial can be characterized
geometrically: 
\eobs

\begin{epr}\label{bh}
A non-closed conformal product $(M,c,D)$ (where $M:=U_1\times U_2$,
$U_1,U_2\subset \R^2$, $c:=[g_1+e^{2f}g_2]$, and $g_i$ is the
Euclidean metric on $U_i$) is 
hyper-Hermitian if and only if $f$ is the real part of a holomorphic 
function in the complex variables $z:=x_1\pm ix_2\in U_1\subset \R^2$
and $w:=x_3\pm ix_4\in U_2$. 
\end{epr}
\begin{proof}
If $f$ is the real part of a holomorphic
function $h(z,w)$, then the metric reads $g=g_1+|e ^h|^2g_2$ and is
thus hyper-Hermitian by \cite{bsmf}, \S7. 

Conversely, assume that $(M,c:=[g_1+e^{2f}g_2],D)$ is
hyper-Hermitian. If $I_1$, $I_2$ denote the $D$-parallel weightless
volume forms 
of the factors, viewed as skew-symmetric endomorphisms on $M$, then
$\pm I_1,\pm I_2$ are $D$-parallel Hermitian structures on $M$. Up to
a choice of signs of $I_1,I_2$ (choice that determines the signs in
the definition of the complex variables $z,w$), one might assume that
$I:=I_1+I_2$ belongs to 
the hyper-Hermitian structure, {\em i.e.}, for any unitary vectors $Y_i\in
V_i$, the orthonormal frame $Y_1,I_1Y_1,Y_2,I_2Y_2$ is positively
oriented.  

Let $J$ be another integrable
Hermitian structure anti-commuting with $I$. We denote by
$X_i:=\partial/\partial x_i$ for $1\le i\le 4$, so $IX_1=X_2$ and
$IX_3=X_4$. Then one can write $JX_1=aX_3+bX_4$ for some functions
$a,b:M\to \R$. Since $J$ is Hermitian we must have $e ^{2f}(a ^2+b^2)=1$
and the anti-commutation with $I$ yields $JX_2=bX_3-aX_4$. A
straightforward computation using the vanishing of the Nijenhuis
tensor of $J$ yields 
\bea 0=N^J(X_1,X_2)&=&(\partial_1b-\partial_2a)J(X_3)-
(\partial_1a+\partial_2b)J(X_4) \\&&
+(\partial_4a+\partial_3b)(aX_3-bX_4)-
(\partial_3a-\partial_4b)(bX_3+aX_4).\eea
This is of course equivalent to the fact that $H:=a+ib$ is holomorphic in
the complex variables $z:=x_1- ix_2$ and $w:=x_3+ix_4$, so $f=-\ln
|H|=Re(-\ln H)$ is the real part of a holomorphic function.

\r

This proposition has a few consequences: 
\bi
\item The examples of 4-dimensional
hyper-Hermitian manifolds constructed by Joyce (cf. \cite{j}, or
\cite{bsmf}, \S7) are actually conformal products.
\item In contrast with the Riemannian case, where a Ricci-flat
  K\"ahler surface is automatically hyper-K\"ahler, there exist
  examples of scalar-flat Einstein-Weyl Hermitian manifolds 
  in dimension 4 which are not hyper-Hermitian. Indeed, it suffices to
  choose $f$ bi-harmonic, but not the real part of a holomorphic
  function in any of the variables $z:=x_1\pm ix_2$ and $w:=x_3\pm
  ix_4$, {\em e.g.} $f(z,w)=Re(z)Re(w)=x_1x_3$. On the other hand the
  complex structure on $M$  
  defined by the complex structures of the factors is clearly $D$-parallel.
\item The hyper-Hermitian non-closed conformal products are exactly
  the non-closed conformal {\em multi-products}, {\em i.e.}, whose adapted
  Weyl structure $D$ leaves more than one pair of complementary
  distributions. The restricted holonomy of these special structures
  is $\C^*\subset CO(4)$, as the following proposition states: 
\ei
\begin{prop}\label{multp}
A $4$-dimensional non-closed conformal product $(M,c,D)$ that admits
multiple $D$-parallel splittings is hyper-Hermitian, and is locally
isomorphic to one of the examples described in Proposition
\ref{bh}. The holonomy of $D$ is then $\C^*$, acting by scalar
multiplication on $\C^2\simeq TM$. 
\end{prop}
\begin{proof} A parallel splitting corresponds to a pair of
  $D$-parallel decomposable weightless forms. If two such splittings
  exist, the factors $M_1,M_2$ of $M$ need to have the same dimension
  (Theorem \ref{main}). Denoting, as before, by $I_1$ and $I_2$ the
  weightless volume forms of the 2-dimensional factors $M_1$
  and $M_2$ respectively, we have a 2-dimensional space of $D$-parallel
  weightless forms, but the only splitting that they define is the
  original one. Therefore, there must be another $D$-parallel 2-form
  $\omega$, orthogonal to both $I_1$ and $I_2$. 

Viewing $\o$ as a skew-symmetric endomorphism, we can distinguish two
cases: Either its square has only one eigenvalue, thus $\o$ defines a
$D$-parallel complex structure $J$, orthogonal to $I_1$ and $I_2$, or
$\o$ defines itself an orthogonal splitting, in which case, the sum of
the two weightless volume forms of those factors defines a
$D$-parallel complex structure $J$ as well. 

By an appropriate change of signs of the weightless volume forms
$I_1,I_2$, we can assume that $I:=I_1+ I_2$ defines the same
orientation as $J$. With respect to this orientation of $M$, the
2-forms $I$ and $J$ are {\em self-dual} and orthogonal, thus they
anti-commute as endomorphisms and define therefore a $D$-parallel
hyper-Hermitian structure on $M$. Proposition \ref{bh} applies and we
have a local model. We shall see below that the Faraday form $F$ is,
in this case, {\em anti-self-dual}. 

For the second claim, we first note that the holonomy representation
of $D$ preserves all self-dual forms and at least one anti-self-dual
2-form, namely $\tilde I:=I_1- I_2$, thus $\Hol_0(D)\subset \C^*$.  

In order to prove that the restricted holonomy of $D$ is exactly
$\C^*$, we will show that the Lie algebra of $\Hol_0(D)$ is
2-dimensional. 

First, denote as before by $X_i,Y_i$ arbitrary vector fields on $M_i$,
lifted to vector fields on $M$. We have 
$$R^D_{X_1,Y_1}X_2=0;\ R^D_{X_2,Y_2}X_1=0,$$
and, since $\Delta_if=0$, the metric $e^{\pm 2f}g_i$ is flat, therefore
$$R^D_{X_i,Y_i}Z_i=0,$$ 
hence 
$$R^D(I_i^\sharp)=0.$$
Here we view the curvature operator $R^D$ as being defined on
$\Lambda^2(TM)$ with values in $\End(TM)$. We will show that the image
of $R^D$ is 2-dimensional in every point where $F$ does not vanish. 

We have used the following notation: $\sharp:T^*M\ra TM$ and its
inverse $\flat:TM\ra T^*M$ are the canonical isomorphisms for a
particular choice of a metric in the conformal class (fixed once and
for all), traditionally called ``raising'', respectively, ``lowering''
of indices. The index $\sharp$, respectively $\flat$, attached {\em
  once} to a tensor, signifies the raising, respectively lowering of {\em
  one} index. For example, $F^\sharp$ is a skew-symmetric endomorphism
of $TM$, and $F^{\sharp\sharp}$ is a bi-vector, {\em i.e.} a section in
$\Lambda^2TM$. 

Straightforward computations show:
\beq\label{asd}
R^D_{X_1,X_2}Y=F(X_1,X_2)Y+(F\wedge
\Id)(X_1,X_2)(Y)=F(X_1,X_2)Y+[F^\sharp,X_1\wedge X_2^\flat](Y),\eeq 
where the bracket $[\cdot,\cdot]$ denotes the commutator in $\End(TM)$.

We infer that the skew-symmetric endomorphism $[F^\sharp,X_1\wedge
X_2^\flat]$ is always anti-self-dual (since it commutes with the basis
of the space of self-dual endomorphisms, defined by the $D$-parallel
hyper-Hermitian structure), which implies that $F$ itself is
anti-self-dual.  

Another consequence of (\ref{asd}) is that
$$R^D(\alpha)=\langle F,\alpha\rangle \Id + [F^\sharp,\alpha^\flat], $$
for any 2-vector $\alpha$ orthogonal on both $I_i^{\sharp\sharp}$,
$i=1,2$. In particular, in a point where $F\ne 0$, 
$$R^D(F)=\|F\|^2 \Id\ \mbox{and}\ R^D(\beta)=[F^\sharp,\beta^\flat]$$
where  $\tilde I^\flat,F,\beta^{\flat\flat}$ is an orthogonal basis of
$\Lambda^-M$. 

Thus $R^D(F)$ and $R^D(\beta)$ are linearly independent at any point where
$F\ne 0$, since the first is a
multiple of the identity and the second is a skew-symmetric
endomorphism. That means that the image through the curvature operator
of $\Lambda^2(TM)$ is 2-dimensional on an open set, hence the holonomy
algebra is 2-dimensional, more precisely $\Hol_0(D)=\C^*$ as claimed. 

\r



\begin{thebibliography}{22}
{

\bibitem{arm} {\sc S. Armstrong}, {\sl Definite signature conformal holonomy: A complete classification}, J. Geom. Phys. {\bf 57}, 2024--2048 (2007).
 
\bibitem{bes} {\sc A. Besse,} {\it
 Einstein manifolds},
 Ergebnisse der Mathematik und ihrer Grenzgebiete (3)  10.
 Springer-Verlag, Berlin, (1987).

\bibitem{calderbank} {\sc D.M.J. Calderbank}, {\sl Selfdual Einstein
    metrics and conformal submersions},\\ arXiv:math/0001041v1 (2000).

\bibitem{toda} {\sc D.M.J. Calderbank}, {\sl The geometry of the Toda equation}
J. Geom. Phys. {\bf 36}, 152--162 (2000).

\bibitem{cabrera} {\sc J.C. Gonzalez-Davila, F. Martin Cabrera,
    M. Salvai}, {\sl Harmonicity of sections of sphere bundles,}
  arXiv:0711.3703v1 (2007).

\bibitem{g} {\sc P. Gauduchon}, {\sl Structures de Weyl-Einstein,
    espaces de twisteurs et vari\'et\'es de type $S^1\times S^3$},
  J. reine angew. Math. {\bf 469}, 1--50 (1995).

\bibitem{j} {\sc D. Joyce,} {\sl Explicit Construction of Self-dual
    4-Manifolds,} Duke Math. J. {\bf 77}, No. 3, 519--552 (1995).

\bibitem{kn2} {\sc S. Kobayashi, K. Nomizu}, {\it Foundations of Differential 
Geometry II}, New York, Interscience Publishers, 1969.


\bibitem{SM} {\sc S. Merkulov, L. Schwachh\"ofer}, {\sl Classification
    of irreducible holonomies of torsion-free affine connections},
  Annals of Mathematics {\bf 150}, No.1, 77--149 (1999).

\bibitem{bsmf} {\sc A. Moroianu,} {\sl Structures de Weyl admettant
    des spineurs parall\`eles}, Bull. Soc. Math. France {\bf 124},
  685--695 (1996).   

\bibitem{pps} {\sc H. Pedersen, Y. S. Poon, A. Swann,} {\sl The
    Einstein-Weyl equations in complex and quaternionic geometry},
  Diff. Geom. Appl. {\bf 3}, 309--321 (1993).



}
\end{thebibliography}
\end{document}